\documentclass[12pt, twoside, leqno]{article}

\usepackage[pdftex]{graphicx}
\usepackage{float}

\usepackage{amsmath,amsthm}
\usepackage{amssymb}
\usepackage{amsbsy}

\usepackage{comment}

\usepackage{enumitem}

\newtheorem{theorem}{Theorem}[section]
\newtheorem{corollary}[theorem]{Corollary}

\newtheorem{proposition}[theorem]{Proposition}

\newtheorem{fact}[theorem]{Fact}

\theoremstyle{definition}
\newtheorem{definition}[theorem]{Definition}
\newtheorem{remark}[theorem]{Remark}

\newtheorem*{xPA}{Problem I}
\newtheorem*{xPB}{Problem II}
\newtheorem*{xMR}{Main Result A}
\newtheorem*{thmB}{Theorem B}
\newtheorem*{thmC}{Theorem C}

\newenvironment{proof*}[1]
  {\renewcommand{\proofname}{#1}%
   \begin{proof}}
  {\end{proof}}

\def\Arg{\text{{\rm Arg }}}

\def\dfrac{\displaystyle\frac}

\def\pmod#1{\left( \mbox{\rm mod~} #1\right)}

\def\Z{\mathbb{Z}}

\numberwithin{equation}{section}

\begin{document}

\baselineskip=17pt

\title{Explicit rank-one constructions for irrational rotations}

\author{Alexandre I. Danilenko \\
B.  Verkin Institute for Low Temperature Physics \& Engineering\\
 Ukrainian National Academy of Sciences\\
47 Nauky Ave.\\
 61164,  Kharkiv, UKRAINE\\
E-mail: alexandre.danilenko@gmail.com
\and 
Mykyta I. Vieprik\\
V. N. Karazin Kharkiv National University\\
4 Svobody sq.\\
 61022, Kharkiv, UKRAINE\\
 E-mail: nikita.veprik@gmail.com}

\date{}

\maketitle

\renewcommand{\thefootnote}{}



\renewcommand{\thefootnote}{\arabic{footnote}}
\setcounter{footnote}{0}

\begin{abstract}
For each {\it well approximable} irrational $\theta$, we provide an explicit  rank-one construction of the $e^{2\pi i\theta}$-rotation $R_\theta$ on the circle $\Bbb T$.
This solves ``almost surely'' a problem by del Junco. 
For {\it every} irrational $\theta$, we construct explicitly a rank-one transformation with an eigenvalue 
$e^{2\pi i\theta}$.
For every irrational $\theta$, two infinite $\sigma$-finite invariant measures $\mu_\theta$ and $\mu_{\theta}'$ on $\Bbb T$ are constructed explicitly such that $(\Bbb T,\mu_\theta, R_\theta)$ is {\it rigid} and of rank one and $(\Bbb T,\mu_\theta', R_\theta)$ is of {\it zero type} and of rank one.
The centralizer of the latter system consists of  just  the powers of $R_\theta$.
Some versions of  the aforementioned results are proved under an extra condition on boundedness of the sequence of cuts in the rank-one construction.
\end{abstract}

\section{Introduction}

By a dynamical system we mean a quadruple $(X,\frak B,\mu, T)$, where $(X,\frak B)$
is a standard Borel space, $\mu$ is a $\sigma$-finite measure on $\frak B$ and $T$ is an invertible $\mu$-preserving transformation of $X$.
The dynamical system (or just $T$) is called  of {\it  rank one} if there is a sequence of finite $T$-Rokhlin towers that approximates the subring  of subsets of finite measure in $\frak B$.
There is an alternative   (explicit) definition of rank one system via an inductive construc\-tion
 process of cutting-and-stacking with a single tower on each step.
It is completely  determined by two underlying sequences of {\it cuts} and  {\it spacer mappings}.
For details and for the equivalence of various definitions of rank-one we refer to \cite{Fe}.

Let $\theta\in(0,1)$ be an irrational number and let $\lambda:=e^{2\pi i\theta}\in\Bbb T$.
Denote by $R_\lambda$ the $\lambda$-rotation on the circle $\Bbb T$.
Del Junco showed in  \cite{dJ2} that $R_\lambda$ is of rank one if $\Bbb T$ is furnished with the Haar measure.
He also raised a related  (more subtle) problem in \cite{dJ1}:

\begin{xPA}\label{prA}
Given $\theta$, provide an {\it explicit} construction (i.e. find  sequences of cuts and spacer mappings) of a rank-one transformation which is isomorphic to $R_\lambda$.
\end{xPA}

A solution of Problem~I for an uncountable subset of well approximable
 irrationals of zero Lebesgue measure was found recently in \cite{Dr--Si}.
 
We  consider also a weak version of Problem~I.

\begin{xPB}\label{prB}
Given $\theta$, provide an {\it explicit} construction  of a rank-one probability preserving transformation $T$
 which  has an eigenvalue $\lambda$.
\end{xPB}

We recall that a number $\lambda\in\Bbb T$ is an eigenvalue of $(X,\frak B,\mu,T)$ if there is a Borel function $f:X\to\Bbb T$ such that $f\circ T=\lambda f$ almost everywhere. 
This implies that $R_\alpha$ is a factor of $T$.
Hence $R_\alpha$ is isomorphic to $T$ if and only if $f$ is one-to-one.
Del Junco   solved Problem~II for a.e. $\theta\in(0,1)$ in \cite{dJ1}.

We now state the main results of the first part  (related to the probability preserving systems) of the present paper.

\begin{xMR}
\begin{itemize}
\item  {\it Problem II is solved for every $\theta$.
\item  Problem I is solved for each  well approximable  $\theta$.
\item For almost all $\theta\in(0,1)$, including the badly approximable reals and the algebraic numbers,
we solve Problem~II in the subclass of rank-one transformations with only two cuts at every step of their inductive construction.}
 \end{itemize}
\end{xMR}

As the subset of well approximable reals from $(0,1)$ is of Lebesgue measure 1,
  Problem~I is solved  ``almost surely''.

  In connection with the third point of Main Result A,  we note that Problem~II (and hence Problem~I)
 can not be solved for any irrational $\theta$ in the subclass of rank-one transformations {\it with bounded parameters\footnote{This means that the number of cuts and the total number of spacers added on the $n$-th step of the construction are both uniformly bounded in $n$.}}, as the eigenvalues of
every such transformation are of finite order (see \cite[Theorem~3]{El--Ru} or
 \cite[Theorem~M]{Da5}).
 We also provide a short alternative proof of this fact.

In the second part of the paper we consider Problem~I  within the class of  the infinite measure preserving dynamical systems.
The main difference from the probability preserving case is that for each irrational $\theta$, there exist {\it uncountably many} mutually disjoint $R_\lambda$-invariant infinite $\sigma$-finite measures on $\Bbb T$ (see, e.g., \cite{Sc}).
Infinite measure preserving {\it rank-one} rotations on $\Bbb T$ were under study  in a recent paper \cite{Dr--Si}.
Construction of the rank-one systems there is based on the same cutting-and-stacking algorithm as in \cite{dJ1} 
but without the spacer growth restriction (to obtain infinite measure). 
For a.e. $\theta\in(0,1)$, an infinite measure $m_\theta$ on $\Bbb T$  was constructed  in \cite{Dr--Si} such that 
the system $(\Bbb T,m_\theta, R_\lambda)$ is  of rank-one with explicit cutting-and-stacking parameters.
Moreover, it was shown that the system is rigid if and only if $\theta$ is well approximable. 
We generalize and sharpen those results in the following two theorems.

\begin{thmB} \label{t1}
Let 
$\lambda\in\Bbb T$ be of infinite order.
Then  there is an infinite $\sigma$-finite  $R_\lambda$-invariant non-atomic Borel measure $\mu_\lambda$ on $\Bbb T$ such that
\begin{itemize}
\item
the dynamical system $(\Bbb T,\mu_\lambda,R_\lambda)$ is  of rank one, 
the parameters of the underlying cutting-and-stacking  construction  are explicitly described,
\item
the number of cuts on each step of the  construction is $2$ and
\item
$(\Bbb T,\mu_\lambda,R_\lambda)$ is  rigid, hence the centralizer $C(R_\lambda)$ of $R_\lambda$ is uncountable.
\end{itemize}
\end{thmB} 

\begin{thmC} \label{t2}
For each element	$\lambda\in\Bbb T$  of infinite order,
  there is an infinite $\sigma$-finite  $R_\lambda$-invariant non-atomic Borel measure $\mu_\lambda'$ on $\Bbb T$ such that
\begin{itemize}
\item
the dynamical system $(\Bbb T,\mu'_\lambda,R_\lambda)$ is  of rank one, 
the parameters of the underlying cutting-and-stacking  construction  are explicitly described,
\item
$(\Bbb T,\mu'_\lambda,R_\lambda)$ is  totally ergodic  and of zero type,
\item
$C(R_\lambda)=\{R_\lambda^n\mid n\in\Bbb Z\}$,
\item
$\mu'_\lambda\circ R_\beta\perp\mu'_\lambda$ whenever 
$\beta\not\in\{\lambda^n\mid n\in\Bbb Z\}$ and
\item
if an element $\omega\in\Bbb T$ is of infinite order with $\omega\not\in\{\lambda,\lambda^{-1}\}$ then
$\mu'_\omega\perp\mu_\lambda'$.
\end{itemize}
\end{thmC}

As far as we know, Theorem~C provides the first examples of {\it spectrally mixing}\footnote{An infinite measure preserving transformation $S$ is of zero type if and only if the  measure of maximal spectral type of $S$ is Rajchman, i.e. the Koopman operator  associated with $S$ is mixing.} ergodic  infinite invariant measures for irrational rotations.

Everywhere below in this paper we construct the rank-one systems via the $(C,F)$-construction.
It was introduced in \cite{dJ3} and \cite{Da1} (in a  different form). 
Various kinds of the $(C,F)$-construction, interrelationship among them and the classical cutting-and-stacking are discussed in detail in \cite{Da2}.

The outline of the paper is as follows. 
Preliminary information from the theory of continuous fractions, dynamical systems and $(C,F)$-construction is collected in \S 2.
In \S 3 we study  eigenvalues and  eigenfunctions of the $(C,F)$-systems.
Main Result A is proved in \S 4.
\S 5 is devoted to the proof of Theorems~B and C.

\section{Preliminaries}

\subsubsection*{Continued fractions}

We recall some basic facts from the theory of continued fractions.
Every irrational number $\theta$ can be represented as an infinite continued fraction 
$[a_0; a_1, a_2, \ldots]$ with $a_j\in\Bbb N$ for each $j>0$.
The rational numbers $\dfrac{p_k}{q_k} := [a_0; a_1, a_2, \ldots, a_k]$ with $(p_k,q_k)=1$ are called the \textit{convergents} for $\theta$.
For each $k\geq 2$,  
\begin{align*}
	p_k &= a_{k-1} q_{k-1} + p_{k-2}, \text{    } p_0 = 1, \ p_1 = a_0, \\
	q_k &= a_{k-1} q_{k-1} + q_{k-2}, \text{    } q_0 = 0, \ q_1 = 1.
\end{align*}

\begin{definition}
	An irrational number $\theta$ is called \textit{badly approximable} if the sequence  $(a_k)_{k=1}^\infty$  is bounded. 
	The irrational numbers that are not badly approximable are called \textit{well approximable}.
\end{definition}

We will utilize the following well known results from the theory of continuous fractions (see \cite{Kh} for the proof).

\begin{fact} \label{f1}
\begin{enumerate}[label=\upshape(\roman*), leftmargin=*, widest=iii]
	\item 
	$(q_n\theta-p_n)(q_{n+1}\theta-p_{n+1})<0$
	for each $n>0$.
	\item
	$
 \dfrac{1}{q_n+q_{n+1}}<|\theta q_n-p_n|<\dfrac 1{q_{n+1}}
 $ for each $n>0$.
	\item
	$|q_n\theta - p_n| < \min\Big\{|b\theta - a|:{\Bbb N\ni b\le q_n, a\in\Bbb Z}, \frac ab\ne\frac{p_n}{q_n}\Big\}$
	 for each $n>0$.
	 \item
	 If there exist $p,q\in\Bbb N$ such that
	 $$
	 |q\theta - p| < \min\Big\{|b\theta - a|:{\Bbb N\ni b\le q, a\in\Bbb Z}, \frac ab\ne\frac{p}{q}\Big\}
	 $$
	 then $\frac pq=\frac{p_n}{q_n}$ for some $n\in\Bbb N$.
	\item 
	$\theta$ is badly approximable if and only if there exists  a real $\delta>0$ such that
	$\min_{p\in\Bbb Z}|q\theta-p|>\frac \delta q$ for each $q\in\Bbb N$.
	\item 
	If $\theta$ is algebraic of power $n$ then  there exists  a real $\delta>0$ such that
	$\min_{p\in\Bbb Z}|q\theta-p|>\frac \delta {q^n}$ for each $q\in\Bbb N$.
	\item
	The set of badly approximable numbers has Lebesgue measure zero.
\end{enumerate}
\end{fact}

\subsubsection*{Dynamical systems}
We recall that a dynamical system $(X,\frak B,\mu,T)$  (or just $T$) is called
\begin{itemize}
\item {\it ergodic}
if each $T$-invariant subset is either $\mu$-null or $\mu$-conull;
\item{\it totally ergodic} 
if $T^p$ is ergodic for each $p\in\Bbb N$.
\item  {\it rigid} if there is a sequence $n_1<n_2<\cdots$ such that
$$
\mu(T^{n_k}A\cap B)\to\mu(A\cap B)\quad\text{as $k\to\infty$}
$$ 
for all subsets $A, B\in\frak B$ with $\mu(A)<\infty$ and $\mu(B)<\infty$;
\item
{\it of $0$-type} if
$$
\mu(T^{n}A\cap B)\to 0\quad\text{as $n\to\infty$}
$$
  for all subsets $A, B\in\frak B$ with $\mu(A)<\infty$ and $\mu(B)<\infty$;
  \item 
  {\it of rank one} if there are subsets $B_1,B_2,\dots$ of $X$ with $\mu(B_k)<\infty$ for each $k\in\Bbb N$ and a sequence of positive integers $n_1<n_2<\cdots$ such that $T^lB_k\cap T^mB_k=\emptyset$ whenever $0\le l<m<n_k$ and $k\in\Bbb N$ and for each subset $A\in\mathfrak B$ of finite measure,
  $$
  \lim_{k\to\infty}\min_{J\subset\{0,\dots, n_k-1\}}\mu\bigg(A\triangle \bigsqcup_{j\in J}T^jB_k\bigg)=0.
  $$
\end{itemize}
Of course, if $T$ is of 0-type then $\mu(X)=\infty$.
If $T$ is of rank one then $T$ is ergodic.

\begin{definition} Suppose that $T$ is ergodic.
A  number $\lambda\in\Bbb T$ is called an {\it eigenvalue of $T$} if there is a measurable function $f:X\to\Bbb T$ such that $f\circ T=\lambda f$.
The function $f$ is called a {\it $\lambda$-eigenfunction} of $T$.
It is defined up to a multiplicative constant from $\Bbb T$.
The set of all eigenvalues is called {\it the $L^\infty$-spectrum} of $T$ and denoted by $e(T)$.
\end{definition}

Of course, $e(T)$ is a subgroup of $\Bbb T$.
If $\mu(X)<\infty$ then $e(T)$ is countable.
It is straightforward to verify that $T$ is totally ergodic if and only if $e(T)$ is torsion free.

For the $\lambda$-rotation $R_\lambda$ on $\Bbb T$ endowed with the Haar measure,
$e(R_\lambda)=\{\lambda^n\mid n\in\Bbb Z\}$.

The {\it centralizer} $C(T)$ of $T$ is the set of invertible $\mu$-preserving transformations that commute with $T$.
Of course, $C(T)$ is a group.

\subsubsection*{$(C, F)$-dynamical systems}

For a detailed exposition of the $(C,F)$-construction of (funny) rank-one actions we refer to \cite{Da1} and \cite{Da2}.
Let $(F_n)_{n\geq 0}$ and $(C_n)_{n\geq 1}$ be two sequences of finite subsets in $\mathbb{Z}$ such that for each $n>0$,
\begin{align}
   \label{2.1}  &F_{0} = \{0\}, \#C_{n} > 1,\\ 
\label{2.2}    &F_{n} + C_{n+1}\subset F_{n+1},\\ 
\label{2.3}     &(F_{n} + c)\cap (F_{n} + c') = \varnothing\text{, if $c, c'\in C_{n+1}$ and $c \neq c'$. }
\end{align}
We let $X_n := F_{n} \times C_{n+1} \times C_{n+2} \times\ldots$ and endow this set with the infinite product topology. Then $X_n$ is a compact Cantor space. The mapping 
\begin{equation*}
    X_n \ni (f_n,c_{n+1},c_{n+2}\ldots) \mapsto (f_n + c_{n+1},c_{n+2},\ldots) \in X_{n+1}
\end{equation*}
is a topological embedding of $X_n$ into $X_{n+1}$. 
Therefore the  inductive limit $X$ of the sequence $(X_n)_{n\geq 0}$ furnished with these embeddings is  well defined.
Moreover, $X$ is a locally compact Cantor space. 
Given a subset $A\subset F_n$, we let
\begin{equation*}
    [A]_n := \{x=(f_n,c_{n+1},\ldots)\in X_n, f_n\in A\}
\end{equation*}
and call this set an $n$-{\it cylinder} in $X$.
 It is open and compact in $X$. 
 For brevity, we will write $[f]_n$ for $[\{f\}]_n, f\in F_n$.
Also, we will write $\mathbf 0$ for $(0,0,\ldots)\in X_0\subset X$.
There exists a unique $\sigma$-finite Borel measure $\mu$ on $X$ such that $\mu(X_0)=1$ and 
\begin{equation*}
    \mu([f]_n)=\mu([f']_n) \text{ for all } f,f'\in F_n, \ n\geq 0.
\end{equation*}
It is easy to verify that
\begin{equation*}
    \mu([A]_n)=\dfrac{\# A}{\# C_1 \cdots \# C_n} \text{ for each subset } A\subset F_n, \ n>0.
\end{equation*}
We also note that $\mu(X)<\infty$ if and only if
\begin{equation}\label{eq2.4}
\sum_{n=1}^\infty \frac{\# F_{n+1}-\# F_n\#C_{n+1}}{\# F_{n+1}}<\infty.
\end{equation}
From now on, 
\begin{equation}
F_n=\{0,1,\dots, h_n-1\}\quad\text{ for some $h_n>0$ and every $n\in\Bbb N$.}
\end{equation}
Then we can   define a transformation $T$ on $X$. 
We first note that for $\mu$-a.e. $x\in X$,
there is $n>0$ such that $x=(f_n,c_{n+1},\ldots)\in X_n$ and $1+f_n\in F_n$. 
We now let
\begin{equation*}
    T x := (1+f_n,c_{n+1},\ldots)\in X_n\subset X.
\end{equation*}
Then $T $ is a well defined  $\mu$-preserving transformation  of $X$. 
We call $(X,\mu, T)$ the $(C,F)$-{\it dynamical system associated with the sequence  
$(C_n,F_{n-1})_{n\geq 0}$.}
This system is of rank one.
We will need the following two facts about $T$.

\begin{fact}\label{fact2.1}
If there is an infinite sequence $n_1<n_2<\cdots$ such that $\# C_{n_k}\to+\infty$ as $k\to\infty$
and $C_{n_k}$ is an arithmetic progression then $T$ is rigid.
\end{fact}

\begin{fact}\label{fact2.2}
If for each $n>0$,
\begin{itemize}
\item
$F_n+F_n+C_{n+1}\subset F_{n+1}$,
\item
the sets $F_n-F_n+c-c'$, $c\ne c'\in C_{n+1}$, and $F_n-F_n$ are all mutually disjoint,
\item
$\# C_n\to+\infty$ as $n\to\infty$
\end{itemize}
then $T$ is of zero type. 
\end{fact}

Fact~\ref{fact2.1} is well known (see, for instance, the proof of \cite[Theorem~0.1]{Da4}).
A proof of Fact~\ref{fact2.2}  can be obtained as a slight modification of the proof of \cite[Theorem~6.1]{Da3}.

We consider also an equivalence relation $\mathcal R$ on $X$:
\begin{equation}\label{ga2.6}
\begin{gathered}
(x,y)\in\mathcal R\iff\exists n>0\text{ such that } x=(f_n,c_{n+1},\dots)\in X_n,\\
 y=(f_n',c_{n+1}',\dots)\in X_n,\text{ and $c_j=c_j'$ for each $j>n$}.
\end{gathered}
\end{equation}
We call $\mathcal R$ {\it the $(C,F)$-equivalence relation} on $X$.
The $T$-orbit equivalence relation coincides with $\mathcal R$ reduced to a $\mu$-conull subset.

\section{Eigenfunctions of $(C,F)$-equivalence relations}

Let $\mathcal{R}$ be the $(C,F)$-equivalence relation on $X$
and let 
(\ref{ga2.6}) hold.
We now define a Borel mapping $d:\mathcal R\to \Z$ by setting 
\begin{equation*}
    d(x,y) := f_n-f_n', \quad\text{for each $(x,y)\in\mathcal R$}.
\end{equation*}
It is straightforward to verify that
 $$
 d(x,y)+d(y,z)=d(x,z)\ \text{ for all $(x,y),(y,z)\in \mathcal R$.}
 $$
 In other words, 
 $d$
   is a Borel cocycle of $\mathcal{R}$
  with values in $\Bbb Z$.

\begin{definition}\label{def3}
	We call a complex number $\lambda\in\Bbb T$  an {\it eigenvalue} of $\mathcal R$ if 
	the cocycle $\mathcal R\ni (x,y)\mapsto\lambda^{d(x,y)}\in\Bbb T$ is a coboundary, i.e.
	there exists a Borel  map $\varphi: X\rightarrow \Bbb T$ and a subset $A\subset X$ such that $\mu(A)=0$ and 
	$$
	\lambda^{d(x,y)}=\varphi(x)\varphi(y)^{-1}\qquad\text{for all $(x,y)\in\mathcal{R}$ with $x,y\not\in A$.}
	$$	 
	We call $\varphi$ a {\it $\lambda$-eigenfunction for $\mathcal R$}. 
	It is defined up to a multiplicative constant.
	We denote by  $e(\mathcal R)$ the set of all eigenvalues of $\mathcal R$.
\end{definition}

\subsubsection*{Continuous eigenfunctions}
In this subsection we study only  continuous eigenfunctions of $\mathcal R$ and the corresponding eigenvalues.

\begin{proposition}\label{pr1}
	Let $\lambda\in e(\mathcal{R})$.
	If a $\lambda$-eigenfunction $\varphi$  is continuous   at a point, then it is continuous everywhere on $X$ and
	for each sequence $(c_k)_{k=1}^\infty\in C_1\times C_2\times\cdots$, the series 
	$\prod\limits_{k=1}^{\infty}\lambda^{c_k}$ converges.
	Moreover, for each  $x=(f_n, c_{n+1}, c_{n+2},\ldots)\in X_n\subset X$, we have that 
	\begin{equation*}
	\varphi(x) = \varphi(\mathbf{0}) \,\lambda^{f_n}\prod\limits_{k=n+1}^{\infty}\lambda^{c_k}.
	\end{equation*}
\end{proposition}

\begin{proof} We will consider only the case where $\varphi$ is continuous at $\mathbf 0$. 
The other cases are considered in a similar way.
Take a point $x\in X$.
Then there is $n>0$ such that $x\in X_n$ and $x=(f_n, c_{n+1}, c_{n+2},\dots)$ for some $f_n\in F_n$ and $c_k\in C_k$ if $k>n$.
We now set 
$$
p_k(x):=(0,\dots,0,c_{k+1},c_{k+2},\dots)\in X_n
$$
for each $k\ge n$.
Then $p_k(x)\to\mathbf 0$ as $k\to\infty$.
Hence $\varphi(p_k(x))\to \varphi(\mathbf 0)$ as $k\to\infty$.
On the other hand, $(x,p_k(x))\in\mathcal R$ and $d(x,p_k(x))=f_n+c_{n+1}+\cdots+c_k$.
We now have
$$
\lambda^{f_n+c_{n+1}+\cdots+c_k}=\lambda^{d(x,p_k(x))}=
\varphi(x)\varphi(p_k(x))^{-1}.
$$
The righthand side of this formula tends to $\varphi(x)\varphi(\mathbf 0)^{-1}$ as $k\to\infty$.
Hence there is a limit of the lefthand side.
Of course, this limit equals  $\lambda^{f_n}\prod_{k>n}\lambda^{c_k}$,
as desired.

It remains to prove that $\varphi$ is continuous at $x$.
Given $z\in\Bbb T$, we  write Arg$\,z=\tau$ if $z=e^{i\tau}$ and $-\pi<\tau\le\pi$.
For each $k\in\Bbb N$, we select  $a_k,b_k\in C_k$ so that 
$$
\max_{c\in C_k}\text{Arg\,}\lambda^c=\text{Arg\,}\lambda^{a_k}
\quad\text{ and }\quad
\min_{c\in C_k}\text{Arg\,}\lambda^c=\text{Arg\,}\lambda^{b_k}.
$$
Since the series $\prod_{k=1}^\infty\lambda^{a_k}$ and $\prod_{k=1}^\infty\lambda^{b_k}$
converge, we can find, for each $\epsilon>0$, a number $N>0$ so that 
$$
\left|\text{Arg\,}\prod_{k=l}^\infty\lambda^{a_k}\right|<\epsilon\quad \text{and} \quad \left|\text{Arg\,}\prod_{k=l}^\infty\lambda^{b_k}\right|<\epsilon\quad\text{ whenever $l>N$.}
$$
Therefore, if $(y_l,y_{l+1},\dots)\in C_{l}\times C_{l+1}\times\cdots$ then 
$$
-\epsilon<\text{Arg\,}\prod_{k=l}^\infty\lambda^{b_k}\le\text{Arg\,}\prod_{k=l}^\infty\lambda^{y_k}\le\text{Arg\,}\prod_{k=l}^\infty\lambda^{a_k}<\epsilon.
$$
Let $f_l:=f_n+c_{n+1}+\cdots+c_l$.
Then $x\in [f_l]_l$ and $ [f_l]_l$ is a compact open neighborhood of $x$.
Hence, for each 
 $y=(f_l,y_{l+1},y_{l+2},\dots)\in [f]_l$,
$$
|\varphi(x)-\varphi(y)|= \left|\prod_{k>l}\lambda^{c_k}-\prod_{k>l}\lambda^{y_k}\right|\le
 \left|\text{Arg\,}\prod_{k>l}\lambda^{b_k}-\text{Arg\,}\prod_{k>l}\lambda^{a_k}\right|<
2\epsilon.
$$
\end{proof}


We note
that the reasoning above proves also the
converse to Proposition~\ref{pr1}.

\begin{proposition}\label{pr2}
If $\lambda\in\Bbb T$ and the series
$\prod\limits_{k=1}^{\infty}\lambda^{c_k}$ converges for each sequence $(c_k)_{k=1}^\infty$ with
$c_k\in C_k$ for every $k>0$ then   a function $\varphi:X\to\Bbb T$ is well defined by the formula
\begin{equation}\label{eqq1}
X\supset X_n\ni x=(f_n, c_{n+1}, c_{n+2},\ldots)\mapsto\varphi(x):=
\lambda^{f_n}\prod\limits_{k=n+1}^{\infty}\lambda^{c_k}.
\end{equation}
This function is continuous on $X$. 
Moreover, $\lambda\in e(\mathcal R)$ and $\varphi$ is a $\lambda$-eigenfunction for $\mathcal R$.
\end{proposition}

\begin{corollary}\label{co1} If \,$\sum_{n=1}^\infty\max_{c\in C_n}|1-\lambda^{c}|<\infty$
then $\lambda\in e(\mathcal R)$ and the function $\varphi$ defined by
 {\rm (\ref{eqq1})} is 
a continuous $\lambda$-eigenfunction of $\mathcal R$.
\end{corollary}
\proof We note that
$$
|1-\lambda^{\sum_{j=1}^n a_n}|\le\sum_{j=1}^n|1-\lambda^{a_j}|
$$
for arbitrary $a_1,\dots,a_n\in\Bbb Z$.
Hence the condition of the corollary and the Cauchy criterion of convergence  yield that the series 
$\prod\limits_{k=1}^{\infty}\lambda^{c_k}$ converges for each sequence $(c_k)_{k=1}^\infty$ with
$c_k\in C_k$ for every $k>0$. 
It remains to apply~Proposition~\ref{pr2}.
\qed

We now provide a sufficient condition for existence of  one-to-one  eigenfunctions of $\mathcal R$.

\begin{proposition} \label{pr3}
Let $\lambda\in\Bbb T$ be of infinite order.
Suppose that for each $n\ge 0$, 
$$
\min_{ f\ne f'\in F_n'} |1-\lambda^{f-f'}|>\sum_{k>n}\max_{c,c'\in C_k}|1-\lambda^{c-c'}|
$$
then $\lambda\in e(\mathcal R)$ and each $\lambda$-eigenfunction is continuous and one-to-one.
\end{proposition}
\proof It follows from the condition of the proposition that
$$
\sum_{k=1}^\infty\max_{c\in C_k}|1-\lambda^{c}|<\infty.
$$
Hence $\lambda\in e(\mathcal R)$ by Corollary~\ref{co1}.
Let $\varphi$ be the $\lambda$-eigenfunction of $\mathcal R$ defined by (\ref{eqq1}).
It is continuous by Corollary~\ref{co1}. 
 Suppose that $\varphi(x)=\varphi(y)$ for some  $x,y\in X$ and $x\ne y$.
 Then there is $n>0$ such that $x,y\in X_n$, $x=(f_n,c_{n+1},c_{n+2},\dots)$,  
 $y=(f_n',c_{n+1}',c_{n+2}',\dots)$ and $f_n\ne f_n'\in F_n$.
 It follows from (\ref{co1}) that
$$
\lambda^{f_n}\prod\limits_{k>n}\lambda^{c_k} =\lambda^{f_n^{\prime}}\prod\limits_{k>n}\lambda^{c_k^{\prime}},\text{ i.e. \ }
\lambda^{f_n^{\prime}-f_n}=\prod\limits_{k>n}\lambda^{c_k-c_k^{\prime}}.
$$
Hence 
$$
|1- \lambda^{f_n^{\prime}-f_n}|=\bigg|1-\prod\limits_{k>n}\lambda^{c_k-c_k^{\prime}}\bigg|\le\sum_{k>n}|1-\lambda^{c_k-c_k'}|.
$$
This contradicts to the condition of the proposition.
Thus, $\varphi$ is one-to-one.
An arbitrary $\lambda$-eigenfunction is also one-to-one because it is proportional to $\varphi$.
\qed

We will need  one more sufficient condition for existence of one-to-one eigenfunctions.
It is a close analogue of Proposition~\ref{pr3}.
Since it is proved in a very  similar way as Proposition~\ref{pr3}, 
we state it without proof.\footnote{We only note that $|\Arg\lambda^s|\ge |1-\lambda^s|$ for each $s\in \Bbb Z$.}

\begin{proposition} \label{pr4} Let $\lambda\in\Bbb T$ be of infinite order.
If for each $n>0$,
$$
\min_{f\ne f'\in F_n}|\Arg \lambda^{f'-f}|>\sum_{k>n}\max_{c,c'\in C_k}
|\Arg \lambda^{c-c'}|
$$
then $\lambda\in e(\mathcal R)$ and each $\lambda$-eigenfunction is one-to-one.
\end{proposition}

Unfortunately, we were unable to figure out if  Proposition~\ref{pr4} is equivalent to~Proposition~\ref{pr3}.

\subsubsection*{Measurable eigenfunctions}
We consider now the general case, provide an eigenvalue criterion and  describe the structure of arbitrary measurable eigenfunctions  of $\mathcal R$.

\begin{proposition}\label{propeigen}
Let $\lambda\in\Bbb T$.
Then $\lambda\in e(\mathcal R)$ if and only if for each $\epsilon>0$, there is $n>0$ such that for every  $m\ge n$, there exists  a subset 
$E_{n,m}\subset C_n+\cdots+C_m$
satisfying the following:
\begin{enumerate}[label=\upshape(\roman*), leftmargin=*, widest=iii]
\item
$\dfrac{\# E_{n,m}}{\# C_n\cdots\# C_m}>1-\epsilon$ and 
\item
$\max_{c,c'\in E_{n,m}}|1-\lambda^{c-c'}|<\epsilon$.
\end{enumerate}
\end{proposition}

\begin{proof} Let $\lambda\in e(\mathcal R)$ and let $\varphi$ be a $\lambda$-eigenfunction.
Then there is  $w\in\Bbb T$ and a subset $A\subset X$ of positive 
measure such that 
\begin{equation}\label{const}
|\varphi(x)-w|<\epsilon/2\qquad\text{
for each $x\in A$.}
\end{equation}
Then we can find $n>0$ and $f\in F_{n-1}$ such that 
$$
\mu(A\cap[f]_{n-1})>(1-\epsilon^2)\mu([f]_{n-1}).
$$ 
Let $E_{n,m}:=\{c\in C_n+\cdots +C_m\mid \mu(A\cap[f+c]_m)>(1-\epsilon)\mu([f+c]_m)\}$.
Since $[f]_{n-1}=\bigsqcup_{c\in C_n+\cdots +C_m}[f+c]_m$, it follows that
$\frac{\# E_{n,m}}{\# C_n\cdots\# C_m}>1-\epsilon$.
Without loss of generality we may assume that $\epsilon<0.1$.
For each pair $c,c'\in E_{n,m}$, we let  $B:=T^{c-c'}(A\cap[f+c]_m)\cap(A\cap[f+c']_m) $.
Then $\mu(B)>0$, $B\subset A$ and $T^{c'-c}B\subset A$.
Select $x\in B$ such that $\varphi(T^{c'-c}x)=\lambda^{c'-c}\varphi(x)$.
This equality and (\ref{const}) yield that
$$
\omega\pm\epsilon/2=\lambda^{c'-c}\omega\pm\epsilon/2.
$$
Hence $\lambda^{c'-c}=1\pm \epsilon$, as desired.

Conversely, suppose that (i) and (ii) are satisfied.
Then we can construct  an infinite  sequence $n_1<n_2<\cdots$ such that for each $k>0$,
\begin{itemize}
\item[(iii)]
$\dfrac{\# E_{n_k,n_{k+1}-1}}{\# C_{n_k}\cdots\# C_{n_{k+1}-1}}>1-\frac 1{2^k}$ and 
\item[(iv)]
$\max_{c,c'\in E_{n_k,n_{k+1}-1}}|1-\lambda^{c-c'}|<\frac 1{2^k}$.
\end{itemize}
We note that 
the mapping
$$
 (c_{n_1}, c_{n_1+1},\dots)\mapsto((c_{n_1}+\cdots +c_{n_2-1}),(c_{n_2}+
 \cdots +c_{n_3-1}),\dots)
$$
is a measure preserving isomorphism of 
the probability space $\Big([0]_{n_1-1}, \frac{\mu\restriction[0]_{n_1-1}}{\mu([0]_{n_1-1})}\Big)$
 onto  the infinite product space $\bigotimes_{k\ge 1}(C_{n_k}+\cdots+C_{n_{k+1}-1}, \tau_k)$, where $\tau_k$ is  the equidistribution on $C_{n_k}+\cdots+C_{n_{k+1}-1}$.
Hence, the Borel-Cantelli lemma and (iii) yield that there is a $\mu$-null subset $Y_0\subset[0]_{n_1-1}$
such that for each
$$
x=(0,c_{n_1},c_{n_1+1},\dots)\in[0]_{n_1-1}\setminus Y_0,
$$
 there is 
$K=K(x)\in\Bbb N$ with
$$
c_{n_k}+\cdots+ c_{n_{k+1}-1}\in E_{n_k, n_{k+1}-1}\qquad\text{for each $k>K$.}
$$
Fix an element  $(0,c_{n_1}',c_{n_1+1}',\dots)\in[0]_{n_1-1}$ with
$
c_{n_k}'+\cdots+ c_{n_{k+1}-1}'\in E_{n_k, n_{k+1}-1}$
for each $k>0$.
We now define a   map $\phi:[0]_{n_1-1}\to\Bbb T$ by setting
$$
\phi(x)=\prod_{k=1}^\infty\lambda^{(c_{n_k}-c_{n_k}')+\cdots+(c_{n_{k+1}-1}-c_{n_{k+1}-1}')}.
$$
It is well defined in view of (iv).
It is straightforward to verify that  if $(x,y)\in \mathcal R$ and $x,y\in [0]_{n_1-1}\setminus Y_0$ then
\begin{equation}\label{cocycle}
\phi(x)=\lambda^{d(x,y)}\phi(y).
\end{equation}
We recall that  the cocycle $d$ is defined at the beginning of this section.
Thus, $\lambda$ is an eigenvalue for the restriction  of $\mathcal R$ to $[0]_{n_1-1}$.
Let $Y$ denote the smallest $\mathcal R$-invariant subset that includes $Y_0$.
Of course, $\mu(Y)=0$.
Then $\phi$ extends to the entire $X$ in such a way that (\ref{cocycle}) holds
for all $(x,y)\in \mathcal R\cap\big( (X\setminus Y)\times (X\setminus Y)\big)$.
Indeed, given $y\in X\setminus Y$, there is (non-unique!) $x\in [0]_{n_1-1}\setminus Y_0$
such that $(x,y)\in\mathcal R$.
We now define $\phi$ at $x$ as $\lambda^{d(x,y)}\phi(y)$.
It is a routine to verify that $\phi$ is well defined and (\ref{cocycle}) holds for $\phi$ on the entire space $X\setminus Y$.
Hence $\lambda\in e(\mathcal R)$.
\end{proof}

\begin{corollary}\label{bound-eigen} Suppose that the sequence $(\# C_n)_{n=1}^\infty$ is bounded.
If $\lambda\in e(\mathcal R)$ then 
 $\max_{c\in C_n}|1-\lambda^c|\to 0$ as $n\to\infty$.
\end{corollary}

We consider now the case where $\lambda$ is of finite order.

\begin{corollary}\label{finiteorder} Let $\lambda\in e(T)$ be of finite order $p$.
Then there exists $n>0$ such that
for every $m\ge n$, there is a subset $C_m^0\subset C_m$ with $\frac{\# C_m^0}{\# C_m}>1-\epsilon$ and $p$ divides $c-c'$ for all $c,c'\in C_m^0$.
\end{corollary}
\begin{proof} We note that $\lambda^p=1$.
 Therefore, for each $m$ and $c,c'\in C_m$, we have that $\lambda^{c-c'}=\lambda^{\widetilde c-\widetilde{c}'}$, where $0\le \widetilde c<p$, $0\le \widetilde c'<p$ and the differences 
 $c-\widetilde c$ and $c'-\widetilde c'$  are divisible by $p$.
Therefore, if $\epsilon>0$ then the equality $\lambda^{c-c'}=1\pm\epsilon$
implies $\lambda^{\widetilde c-\widetilde c'}=1\pm\epsilon$.
As the set $\{\lambda^{\widetilde c-\widetilde c'}\mid 0\le \widetilde c,c'<p\}$ is finite and $\epsilon$
is arbitrarily small, 
we obtain  that  $\lambda^{\widetilde c-\widetilde c'}=1$, i.e.
$p|(\widetilde c-\widetilde c')$ and hence  $p|(c-c')$ if $\epsilon$  is small enough.
It remains to apply  Proposition~\ref{propeigen}.
\end{proof}

\begin{remark}\label{rem3.7} Let $(X,\frak B,\mu, T)$ be a $(C,F)$-dynamical system and let $\mathcal R$ be the $(C,F)$-equivalence relation on $X$.
Then it is straightforward to verify that $e(T)=e(\mathcal R)$.
\end{remark}

\section{$(C,F)$-systems with finite invariant measure and their eigenfunctions}

In this section we consider  rank-one dynamical systems with finite invariant measure.

\subsubsection*{Solution of Problem II}

Let $\epsilon>0$.
We say that a finite subset $P\subset\Bbb T$ is an {\it $\epsilon$-net} if
$$
\max_{z\in \Bbb T}\min_{p\in P}|z-p|<\epsilon.
$$

\begin{theorem}\label{th41} Given an element $\lambda\in\Bbb T$ of infinite order, there is an explicit $(C,F)$-construction  of  a rank-one transformation $T$
with finite invariant measure and $\lambda\in e(T)$.
 \end{theorem}

\begin{proof}
It is well known (and easy to verify) that for each $n\in\Bbb N$, there is a number $j_n\in\{1,\dots,n\}$ such that $\delta_n:=|1-\lambda^{j_n}|<\frac{2\pi}{n}$.
Hence the  subset
$$
P(\lambda,n):=\bigg\{\lambda^{j_nk}\mid 1\le k\le \frac{2\pi}{\delta_n}\bigg\}\subset\Bbb T
$$
 is a $\frac{2\pi}{n}$-net.
 
 We now construct inductively a sequence $(C_n,F_{n-1})_{n=1}^\infty$ of subsets in $\Bbb Z$.
 We recall that $F_0=\{0\}$ and $F_n=\{0,1,\dots, h_n-1\}$ for some $h_n\in\Bbb N$ and each $n\in\Bbb N$.
Suppose we have already determined $(C_k, F_k)_{k=1}^{n-1}$ for some $n>0$.
Suppose, in addition,  that an auxiliary condition 
\begin{equation}\label{eq4.1}
h_{n-1}>\frac{n^4}{\delta_{n^2}}
\end{equation}
is satisfied.
Our purpose is to define $C_n$ and $F_n$ (or, equivalently, $h_n$).
Since $\lambda^{h_{n-1}}\cdot P(\lambda,n^2)$ is  a $\frac{2\pi}{n^2}$-net, there is   a positive integer $k_1\le \frac{2\pi}{\delta_{n^2}}$ such that  
$$
|1-\lambda^{h_{n-1}+k_1j_{n^2}}|<\frac{2\pi}{n^2}.
$$
We now set $a(1):=h_{n-1}+k_1j_{n^2}$.
In a similar way,  there is   a positive integer $k_2\le \frac{2\pi}{\delta_{n^2}}$ such that
$$
|1-\lambda^{a(1)+h_{n-1}+k_2j_{n^2}}|<\frac{2\pi}{n^2}.
$$
We now set $a(2):=a(1)+h_{n-1}+k_2j_{n^2}$.
Continuing this process  infinitely many times we obtain an infinite  sequence $a(1), a(2), \dots$
such that 
\begin{gather}
\label{eq4.2}
a(l)=a(l-1)+h_{n-1}+k_{l}j_{n^2}\ \text{ for some positive  
$k_{l}\le \frac{2\pi}{\delta_{n^2}}$ and}\\
\label{eq4.3}
|1-\lambda^{a(l)}|<\frac{2\pi}{n^2}
\end{gather}
for each $l>0$.
Let $r_n$ be the smallest $l$ such that $a(r_n-1)+h_{n-1}>\frac{(n+1)^4}{\delta_{(n+1)^2}}$.
Then we set
\begin{align*}
 C_n&:=\{0,a(1), a(2),\dots, a(r_n-1)\}\quad\text{and}\\
 h_n&:= a(r_n-1)+h_{n-1}.
 \end{align*}
Thus, we defined $C_n$ and $F_n$.
Moreover,  (\ref{eq4.1}) holds if we replace $n-1$ with $n$. 
Continuing this process infinitely many times, we
obtain the entire sequence $(C_n,F_n)_{n=1}^\infty$.
It is straightforward to check that (\ref{2.1})--(\ref{2.3}) are satisfied for this sequence.
Hence the associated $(C,F)$-dynamical system $(X,\frak B,\mu, T)$ is well defined.
It is of rank one.

We now verify that $\mu(X)<\infty$.
Indeed, applying (\ref{eq4.1}) and (\ref{eq4.2}) and using the fact that $j_{n^2}\le n^2$ we obtain that
$$
\frac{h_n-h_{n-1}r_n}{h_n}=\frac{(k_1+\cdots+k_{r_n-1})j_{n^2}}{h_n}\le
\frac{2\pi r_n n^2}{\delta_{n^2}h_{n-1} r_n}<\frac{2\pi}{n^2}.
$$
Hence $\sum_{n=1}^\infty\frac{h_n-h_{n-1}r_n}{h_n}<\infty$.
By (\ref{eq2.4}), $\mu(X)<\infty$, as desired.

It remains to check that $\lambda\in e(T)$.
It follows from (\ref{eq4.3}) that
$$
\sum_{n=1}^\infty\max_{c\in C_n}|1-\lambda^c|<\sum_{n=1}^\infty\frac{2\pi}{{n^2}}<\infty.
$$
Hence, Corollary~\ref{co1} and Remark~\ref{rem3.7} yield that $\lambda\in e(T)$.
\end{proof}

Thus, Problem~II
 is solved.

\subsubsection*{Boundedness of the number of  cuts}
In this subsection we refine Theorem~\ref{th41} for a.e. $\lambda\in\Bbb T$.

Given  $\lambda\in\Bbb T$  and  $n\in\Bbb N$, 
we let 
\begin{gather*}
\delta_n(\lambda):=\min_{j\in\{1,\dots,n\}}|1-\lambda^j|\qquad\text{and}\\
E:=\bigcup_{N\in\Bbb N}\bigcap_{n\in\Bbb N}\Big\{\lambda\in\Bbb T\mid \frac{n^4}{N2^{n-1}}<\delta_{n^2}(\lambda)\Big\}.
\end{gather*}

Of course, if $\lambda\in E$ then $\lambda$ is of infinite order in $\Bbb T$.

\begin{theorem}\label{th42} Given $\lambda\in E$, there is an explicit $(C,F)$-construction  of  a rank-one finite measure preserving transformation $T$
such that  $\lambda\in e(T)$ and $\# C_n=2$ for each $n\in\Bbb N$.
 \end{theorem}

\begin{proof}
Let
$\lambda\in\bigcap_{n\in\Bbb N}\{\lambda\in\Bbb T\mid \frac{n^4}{N2^{n-1}}
<\delta_{n^2}(\lambda)\}$ for some $N\in\Bbb N$.
We then repeat the construction in the proof of Theorem~\ref{th41}  almost verbally
but replace the ``stopping time'' condition (\ref{eq4.1}) with the following one:
\begin{equation}\label{eq4.4}
h_{n-1}>N2^n.
\end{equation}
We use below the same notation as in the proof of Theorem~\ref{th41}.
Since $a(1)=h_{n-1}+k_1j_{n^2}$, it follows that $a(1)+h_{n-1}>2h_{n-1}>N2^{n+1}$.
Hence $r_n=2$.
We recall that $\# C_n=r_n$.
Utilizing (\ref{eq4.4}) and the definition of $\lambda$ we obtain that
$$
\frac{h_n-h_{n-1}r_n}{h_n}<\frac{k_1n^2}{2h_{n-1}}<\frac{\pi n^2}{\delta_{n^2(\lambda)}h_{n-1}}
<\frac{\pi}{2n^2}.
$$
Hence $\sum_{n=1}^\infty\frac{h_n-h_{n-1}r_n}{h_n}<\infty$.
Therefore, if $(X,\mu, T)$ stands for the associated $(C,F)$-dynamical system then $\mu(X)<\infty$.
We also have that $|1-a(1)|<\frac{2\pi}{n^2}$.
Therefore
$$
\sum_{n=1}^\infty\max_{c\in C_n}|1-\lambda^{c}|=\sum_{n=1}^\infty|1-\lambda^{c_n}|
<\sum_{n=1}^\infty\frac{2\pi}{{n^2}}<\infty,
$$
where $c_n$ is determined by the equality $C_n=\{1,c_n\}$. 
 Hence $\lambda\in e(T)$.
 Thus, the theorem is proved completely.
\end{proof}

Our next purpose is to show that $E$ is ``large''.
Denote by $\tau$ the Haar measure on $\Bbb T$.

\begin{proposition} $\tau(E)=1$.
\end{proposition}

\begin{proof}
Let
$$
D_n:=\Big\{z\in\Bbb T\mid\exists k\in \Bbb N, k\le n^2\text{ with }|z^k-1|
<\frac{n^4}{2^{n-1}}\Big\}.
$$
Since the  map $\Bbb T\ni z\mapsto z^k\in\Bbb T$ preserves $\tau$, we obtain that
$$
\tau(D_n)\le\sum_{k=1}^{n^2}\tau\Big(\Big\{z\in\Bbb T:|z^k-1|<\frac{n^4}{2^{n-1}}\Big\}\Big)\le
2\pi\sum_{k=1}^{n^2}\frac{n^4}{2^n}<2\pi\frac{n^6}{2^n}.
$$
Hence $\sum_{n=1}^\infty\tau(D_n)<\infty$.
The Borel-Cantelli lemma implies that $\tau$-a.e. $z\in\Bbb T$ there is $M>0$ such that
$z\not\in D_n$ for each $n>M$.
Thus, 
$ \frac{n^4}{2^{n-1}}<\delta_{n^2}(z)$ for each $n>M$.
Of course, $\delta_{n^2}(z)>0$ for each $n\in\Bbb N$
if $z$ is of infinite order in $\Bbb T$.
Hence there exists  $N>0$ such that 
$ \frac{n^4}{N2^{n-1}}<\delta_{n^2}(z)$ for each $n\in\Bbb N$, i.e. $z\in E$.
\end{proof}

\begin{proposition} 
If an irrational $\theta\in(0,1)$ is badly approximable or algebraic then $e^{2\pi i\theta}\in E $.
\end{proposition}

\begin{proof} It follows from Fact~\ref{f1}(v) that if $\theta$ is badly approximable then there
 is a real $d>0$ such that
$|e^{2\pi i\theta q}-1|>\frac{d}{q}$ for each $q\in\Bbb N$.
Hence $\delta_{n^2}(e^{2\pi i\theta})>\frac d{n^2}$.
This  yields that  $e^{2\pi i\theta}\in E$, as desired.

A similar argument  ``works'' for an algebraic  $\theta$ if one refers to 
Fact~\ref{f1}(vi).
\end{proof}

The  $(C,F)$-parameters $(C_n,F_{n-1})_{n=1}^\infty$ are called {\it bounded} if the sequences
$(\# C_n)_{n=1}^\infty$ and $(\# F_{n}-\# F_{n-1}\#C_n)_{n=1}^\infty$ are both bounded 
(see \cite{El--Ru},  \cite{Ry}, \cite{Da5}).
It follows from the structural theorems \cite[Theorem~3]{El--Ru} or
 \cite[Theorem~M]{Da5} that if $T$ is a $(C,F)$-transformation with bounded parameters then every
 $\lambda\in e(T)$ is of finite order in $\Bbb T$.
 We now provide a short direct  proof of this fact.

\begin{proposition} If $T$ is a $(C,F)$-transformation with bounded parameters
and $\lambda\in e(T)$ then $\lambda$ is of finite order.
\end{proposition}

\begin{proof}
We prove this by contraposition.
Suppose that $\lambda\in e(T)$ is of infinite order
and
 $$
 M:=\sup_{n\in\Bbb N}\max(\#C_n,\#F_n - \#F_{n-1}\#C_n)<\infty.
$$
 For each $n\in\Bbb N$, we denote  by $c_n$ the least positive element of $C_n$.
 Then, of course, 
$
   0\le  c_{n+1} - \#F_{n}  \leq M.
$
Therefore,
\begin{align*}
|c_{n+1} - c_n\#C_n| &\le
|c_{n+1} - \#F_n| + |\#F_n - \#F_{n-1}\#C_n| + \#C_n|\#F_{n-1} - c_n| \\
&\le M + M + M^2.
\end{align*}
Let 
$
\delta := \min \{|\lambda^k - 1| :  k = 1,2,\ldots, 2M+M^2\}.
$
Then $\delta>0$ and
for each $n\in\Bbb N$,
\begin{equation}\label{eqr}
|\lambda^{c_{n+1}} - \lambda^{c_n\#C_n}| 
=|\lambda^{c_{n+1} - c_n\#C_n} - 1| \geq \delta.
\end{equation}
On the other hand, 
$
  \lim_{n\to\infty}  |\lambda^{c_n} - 1|=0
$
by Corollary~\ref{bound-eigen}.
Passing to the limit in (\ref{eqr}), we obtain  that 
$|\lambda^{c_{n+1}} - \lambda^{c_n\#C_n}|\to 0$, a contradiction.
\end{proof}

\subsubsection*{Solution of Problem I for the well approximable irrationals}
This subsection is devoted entirely to the proof of the following theorem.

\begin{theorem} Let $\lambda=e^{2\pi i\theta}$ for a   well approximable irrational $\theta\in(0,1)$.
There is an explicit $(C,F)$-construction of a rank-one  probability preserving dynamical 
system $(X, \frak B,\mu,T)$
such that $\lambda\in e(T)$ and the corresponding $\lambda$-eigenfunctions are one-to-one continuous maps from $X$ to $\Bbb T$.
\end{theorem}

\begin{proof}
Let $\theta=[a_0;a_1,\dots]$ stand for the expansion of $\theta$ into the continued fraction.
Let $(\frac{p_k}{q_k})_{k=1}^\infty$   be the sequence of  convergents for $\theta$.
 It follows from Fact~\ref{f1}(ii) that 
  \begin{equation}\label{ineq4.5}
 \frac{|\theta q_n-p_n|}{|\theta q_{n+1}-p_{n+1}|}>\frac{q_{n+2}}{q_n+q_{n+1}}
 =\frac{a_{n+2}q_{n+1}+q_n}{q_{n+1}+q_n}\ge\max\Big(\frac{a_{n+2}}2,1\Big)
  \end{equation}
  for each $n\in\Bbb N$.
Since 
\begin{equation}\label{eqdop}
|\theta q_n-p_n|\to 0 \qquad\text{as $n\to\infty$},
\end{equation}
we obtain that 
 \begin{equation}\label{equ4.6}
|\Arg \lambda^{q_n}|=|\Arg e^{2\pi i(\theta q_n-p_n)}|=2\pi|\theta q_n-p_n|.
 \end{equation}
As $\theta$ is well approximable, the sequence $(a_n)_{n=1}^\infty$ is unbounded.
Hence, in view of (\ref{ineq4.5})--(\ref{equ4.6}), there exists an infinite sequence $m_1<m_2<\cdots$ of positive integers such that
\begin{align}
\label{ineq4.6}
  \dfrac{|\Arg \lambda^{q_{m_k-1}}|}{|\Arg \lambda^{q_{m_k}}|} &> \max\{2 k^2, 4\}\text{  and }\\
  \label{final}
  |\Arg \lambda^{q_{m_{k-1}}}| &>|\Arg \lambda^{q_{m_k-1}}|
  \end{align}
 We have to construct a sequence $(C_n,F_{n-1})_{k=1}^\infty$ satisfying (\ref{2.1})--(\ref{2.3}).
As above,  $F_0=\{0\}$ and $F_n=\{0,1,\dots,h_n-1\}$ for some $h_n> 0$. 
We now let $h_n:=q_{m_n}$ for each $n\in\Bbb N$.
Thus, the sequence $(F_n)_{n=0}^\infty$ is determined completely.
It remains to construct $(C_n)_{n=1}^\infty$.
This will be done inductively.
Suppose that we  already defined $(C_k)_{k=1}^{n-1}$ for some $n\in\Bbb N$.
Our purpose is to specify $C_n$.
For that, we first construct an auxiliary sequence $(b(j))_{j=0}^\infty$ of positive integers.
Let $b(0):=0$.
The other terms of this sequence will be specified  in an inductive way.
Suppose that we already have $b(1),\dots,b(N)$  for  some $N > 0$. 
 If
$$
|\Arg \lambda^{b(N)}| < \dfrac{1}{2}|\Arg \lambda^{q_{m_{n-1}-1}}|
$$
then we call  $N$ {\it good for $b$} and set $b(N+1) := b(N) + h_{n-1}$.  
  If
$$
|\Arg \lambda^{b(N)}| \ge \dfrac{1}{2}|\Arg \lambda^{q_{m_{n-1}-1}}|
$$
then we call  $N$ {\it bad for $b$} and set  $b(N+1) := b(N) + q_{m_{n-1}-1} + 2h_{n-1}$.
Continuing this process infinitely many times we specify the entire sequence 
$(b(j))_{j=0}^\infty$.
Of course, $b(0)<b(1)<\cdots$ and hence $b(j)\to\infty$ as $j\to\infty$.
Let $r_n$ be the  greatest  $j>0$ such that $b(j-1)+h_{n-1}< h_n$.
Then we set
$$
C_n:=\{b(0),\dots, b(r_n-1)\}.
$$
Repeating this construction infinitely many times, we obtain the infinite sequence $(C_k)_{k=1}^\infty$.
It follows from the construction
that (\ref{2.1})--(\ref{2.3}) are satisfied for  $(C_k,F_{k-1})_{k=1}^\infty$.
Hence a $(C,F)$-dynamical system $(X,\frak  B,\mu, T)$ associated with this sequence is well defined.
We have to show that:
\begin{enumerate}[label=\upshape\bf{(\roman*)}, leftmargin=*, widest=iii]
\item
$\mu(X)<\infty$,
\item
$\lambda\in e(T)$ and
\item
the $\lambda$-eigenfunctions of $T$ are continuous and one-to-one.
\end{enumerate}
Prior to this, we 
 fix $n\in\Bbb N$ and analyze some properties of $C_n$ and $F_n$.
For brevity, we will write $m$ for $m_{n-1}$.
To be specific, we assume that $\Arg \lambda^{q_m} > 0$.
(The other case is considered in a similar way.)
Then, in view of  Fact~\ref{f1}(i), $\Arg \lambda^{q_{m-1}} < 0$. 
Denote by $M_1<M_2<\cdots$ the sequence of natural numbers that are bad for $b$. 
We also set $M_0:=-1$.
Then  
$(\Arg\lambda^{b(j)})_{j=M_l+1}^{M_{l+1}}$ is an arithmetic progression with common difference
$\Arg\lambda^{q_m}$ for each $l\ge 0$.
According to our construction,
\begin{gather*}
\max_{0\le j< M_1}\Arg\lambda^{b(j)}<-\frac12\Arg\lambda^{q_{m-1}},\\
-\frac12\Arg\lambda^{q_{m-1}}\le\Arg \lambda^{b(M_1)}<-\frac12\Arg\lambda^{q_{m-1}}+\Arg\lambda^{q_m},\\
\frac12\Arg\lambda^{q_{m-1}}+2\Arg\lambda^{q_m}\le\Arg \lambda^{b(M_1+1)}<\frac12\Arg\lambda^{q_{m-1}}+3\Arg\lambda^{q_m}
\end{gather*}
and so on.
It follows that
for each $j\ge 0$,
$$
    \Arg \lambda^{b(j)} \in \Big[\dfrac{1}{2}\Arg \lambda^{q_{m-1}} + 
    2\Arg \lambda^{q_{m}}, -\dfrac{1}{2}\Arg \lambda^{q_{m-1}} + \Arg \lambda^{q_{m}}\Big).
$$
We illustrate this by  Figure~1 below.

\begin{figure}[H]
\centerline{\includegraphics[width=14cm,height=3.6cm, angle=0]{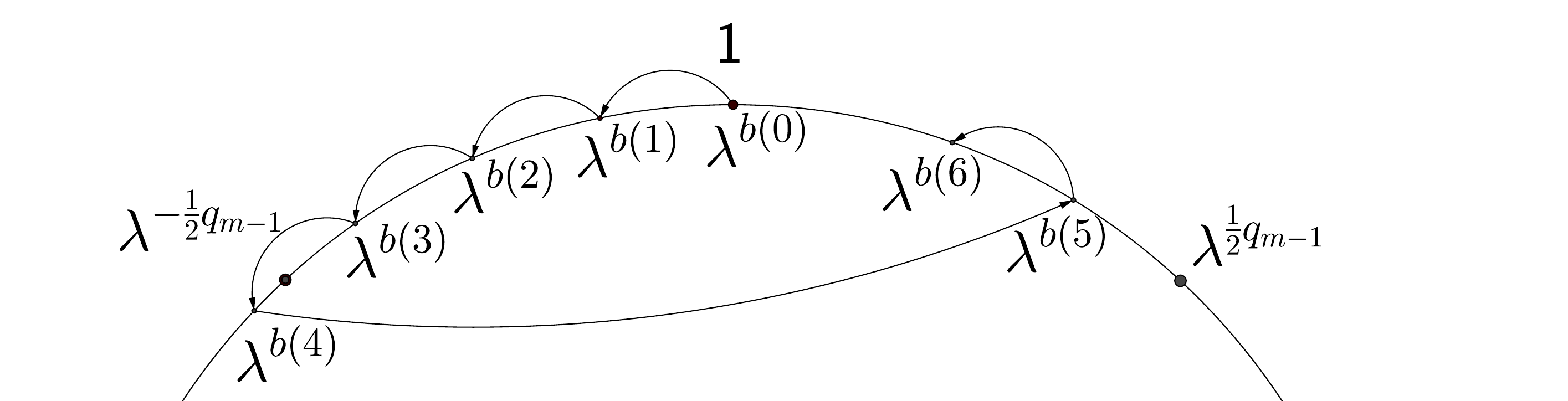}}
\caption{$i = 1, 2, 3, 5, 6$ are good for $b$, and $i = 4$ is bad for $b$.}
\label{fig}
\end{figure}

Hence for all $c,c'\in C_n$,
\begin{equation} \label{aux_con3}
|\Arg\lambda^{c-c'}|=|\Arg \lambda^c-\Arg \lambda^{c
'}|< |\Arg \lambda^{q_{m-1}}|-|\Arg\lambda^{q_{m}}|.
\end{equation}
Applying (\ref{ineq4.6}), we also obtain that for each $j>0$, 
\begin{equation*}\label{eqqM}
\begin{aligned}
M_{j+1}-M_j &\ge\dfrac{\Big(-\dfrac{1}{2}\Arg \lambda^{q_{m}-1}\Big) -
 \Big(\dfrac{1}{2}\Arg \lambda^{q_{m-1}} + 3\Arg \lambda^{q_m}\Big)}{\Arg \lambda^{q_m}} \\
&=\dfrac{-\Arg \lambda^{q_{m-1}}}{\Arg \lambda^{q_m}} - 3  \\
&=\dfrac{|\Arg \lambda^{q_{m-1}}|}{|\Arg \lambda^{q_m}|} -3  \\
&>2(n-1)^2 - 3  \\
&>(n-1)^2
\end{aligned}
\end{equation*}
if $n\ge 3$. 
In a similar way,
$$
   M_1\ge  \dfrac{-\dfrac{1}{2}\Arg \lambda^{q_{m-1}}}{\Arg \lambda^{q_m}} = 
  \dfrac{|\Arg \lambda^{q_{m-1}}|}{2|\Arg \lambda^{q_m}|} > 
    (n-1)^2.
$$
Thus, there are at least $(n-1)^2$ good  integers (for $b$) between every two subsequent bad ones,  and  there are at least $(n-1)^2$ good  integers before the first bad one.
Hence
\begin{equation}\label{eqq}
h_n\ge Bq_m(n-1)^2,
\end{equation}
where  $B$ is the number of bad integers  that are less than $r_n$.
On the other hand,
$$
 h_{n}-h_{n-1}r_n= (q_m+q_{m-1})B+ (h_n-b(r_n-1)-h_{n-1}).
$$
By the definition of $r_n$,
$$
b(r_n-1)+h_{n-1} < h_n\le b(r_n)+h_{n-1}.
$$
Hence $h_n-b(r_n-1)-h_{n-1}\le  b(r_n)-b(r_n-1)\le q_{m-1}+2q_m$.
Therefore
$$
 h_{n}-h_{n-1}r_n\le (q_m+q_{m-1})B+q_{m-1}+2q_m.
$$
This inequality and (\ref{eqq})
yield that
\begin{equation}\label{eq4.9}
\begin{aligned}
\frac{ h_{n}-h_{n-1}r_n}{h_n}&\le\frac{(q_m+q_{m-1})B+q_{m-1}+2q_m}{B(n-1)^2q_m}\\
&\le\frac{2Bq_m+3Bq_m}{B(n-1)^2q_m}\\
&=\frac 5{(n-1)^2}.
\end{aligned}
\end{equation}
If
 $f, f'\in F_n$ then
$
    -q_{m_n} <f_n^{\prime}-f_n <q_{m_n} 
$
and
$$
\min_{f\ne f'\in F_n}|\Arg\lambda^{f-f'}|=\min_{0<k<q_{m_n}}|\Arg e^{2\pi ik\theta}|.
$$
For each positive $k<q_{m_n}$, there is $l_k\ge 0$ such that $|\Arg e^{2\pi ik\theta}| =2\pi|k\theta-l_k|$.
We claim that 
$
|k\theta-l_k|\ge |q_{m_n-1}\theta-p_{m_n-1}|.
$
Indeed, if this inequality does not hold then there is $k\in\{1,\dots, q_{m_n}-1\}$ such that
$$
|k\theta-l_k|< |q_{m_n-1}\theta-p_{m_n-1}|.
$$
It follows from Fact~\ref{f1}(iii) that $k>q_{m_n-1}$.
We now select 
$$
q\in\{q_{m_n-1}+1,\dots, q_{m_n}-1\}
$$
 such that
\begin{equation}\label{extra}
|q\theta-l_q|=\min_{q_{m_n-1}+1<k<q_{m_n}-1}|k\theta-l_k|<|q_{m_n-1}\theta-p_{m_n-1}|.
\end{equation}
Consider two cases.
If $q_{m_{n-1}}<a\le q$ and $b\in\Bbb Z$ then
$$
|a\theta-b|\ge |a\theta-l_a|\ge |q\theta-l_q|.
$$
Moreover, $|a\theta-b|= |q\theta-l_q|$ if and only if $a=q$ and $b=l_q$ and hence $\frac ba=\frac{l_q}q$.

 If $q_{m_{n-1}}\ge a\ge 1$ and $b\in\Bbb Z$ then, in view of Fact~\ref{f1}(iii) and (\ref{extra}),
 $$
|a\theta-b|\ge|q_{m_n-1}\theta-p_{m_n-1}|>|q\theta-l_q|.
$$
Therefore we deduce from  Fact~\ref{f1}(iv) that $\frac {l_q}q$ is a convergent for $\theta$.
This contradicts to the fact that 
$q_{m_n-1}<q< q_{m_n}$.
Thus, we proved that
$$
|k\theta-l_k|\ge |q_{m_n-1}\theta-p_{m_n-1}|.
$$
Therefore,
$$
  2\pi|k\theta-l_k|\ge 2\pi|q_{m_{n-1}}\theta-p_{m_n-1}|=|\Arg \lambda^{q_{m_n-1}}|.
    $$
Hence,
\begin{equation}\label{ineq4.11}
\min_{f\ne f'\in F_n}|\Arg\lambda^{f-f'}|\ge |\Arg \lambda^{q_{m_n-1}}|.
\end{equation}

We are now  ready to verify {\bf (i)}--{\bf (iii)}.
It follows from (\ref{eq4.9}) that 
$$
\sum_{n=1}^\infty\frac{ h_{n}-h_{n-1}r_n}{h_n}<\infty.
$$
Hence $\mu(X)<\infty$ by (\ref{eq2.4}).
Next, for each $n\in\Bbb N$, we utilize (\ref{aux_con3}) infinitely many times to obtain
\begin{equation}\label{port}
\sum_{k>n}\max_{c,c'\in C_k}|\Arg \lambda^{c-c'}|\le\sum_{k>n} (|\Arg \lambda^{q_{m_{k-1}-1}}|
-|\Arg\lambda^{q_{m_{k-1}}}|).
\end{equation}
Since (\ref{final}) holds, it follows that 
$$
\xi:=\sum_{k\ge n} (|\Arg \lambda^{q_{m_k}}|-|\Arg\lambda^{q_{m_{k+1}-1}}|)>0.
$$
Therefore
$$
\sum_{k\ge n} (|\Arg \lambda^{q_{m_k-1}}|-|\Arg\lambda^{q_{m_k}}|)=|\Arg \lambda^{q_{m_n-1}}|-
\xi<|\Arg \lambda^{q_{m_n-1}}|.
$$
This inequality, (\ref{port}) and (\ref{ineq4.11}) yield that
$$
\min_{f\ne f'\in F_n}|\Arg \lambda^{f'-f}|> \sum_{k>n}\max_{c,c'\in C_k}|\Arg \lambda^{c-c'}|.
$$
Hence, by Proposition~\ref{pr4}, 
$\lambda\in e(T)$ and each $\lambda$-eigenfunction
is continuous and one-to-one.
\end{proof}

\section{$(C,F)$-systems with infinite invariant measure and irrational rotations}

\begin{proof}[Proof of Theorem  B]

We define a sequence $(C_k,F_{k-1})_{k\ge 1}$  inductively. 
As in the previous section, we  assume that $F_0 := \{0\}$ and
$F_{k}=\{0,1,\dots,h_{k}-1\}$ for some $h_{k}>0$ for every $k\in\Bbb N$.
Suppose that we have already constructed $(C_k)_{k=1}^{n-1}$ and $(h_k)_{k=1}^{n-1}$ for some $n>0$.
Our purpose is to define  $C_{n}$ and $h_n$.
Since $\lambda^{h_{n-1}}$ is of infinite order in $\Bbb T$,
we can  select $q_n>1$  so that
\begin{equation}\label{first}
\min_{1\le k\le n}\min_{f\ne f'\in F_k}|\lambda^{f-f'} - 1|> 2^n\max_{|j|<2^n}|1-\lambda^{jq_nh_{n-1}}|.
\end{equation}
We now set
\begin{align*}
C_n &:= \{0,q_{n}h_{n-1},2q_{n}h_{n-1},\dots, (2^n-1)q_{n}h_{n-1}\}\text{ \ and }\\
h_n&:=2^nq_{n}h_{n-1}.
\end{align*}
Continuing this procedure infinitely times we define the entire sequences
$(h_k)_{k\ge 0}$  (and hence $(F_k)_{k\ge 0}$) and $(C_k)_{k\ge 1}$.
It is straightforward to verify that~(\ref{2.1})--(\ref{2.3}) are satisfied for these sequences.
Hence the associated $(C,F)$-system $(X,\frak B,\mu,T)$ is well defined.
As $h_n>2 h_{n-1}\# C_n$, it follows from (\ref{eq2.4}) that $\mu(X)=\infty$.
Since $C_n$ is an arithmetic progression and $\# C_n\to\infty$, we deduce from 
Fact~\ref{fact2.1} that $T$ is rigid.
For each $n\in\Bbb N$, it follows from (\ref{first}) that
$$
\sum_{k>n}\max_{c,c'\in C_k}|1-\lambda^{c-c'}|<\sum_{k>n}\frac 1{2^k}\min_{f\ne f'\in F_n}|\lambda^{f-f'} -1|< \min_{f\ne f'\in F_n}|\lambda^{f-f'} -1|.
$$
Then
 Proposition~\ref{pr3} yields that $\lambda\in e(T)$
and each $\lambda$-eigenfunction $f$ of $T$ is one-to-one.
We now set $\mu_\lambda:=\mu\circ f^{-1}$.
Then $f$ is a measure preserving isomorphism of $(X,\mu, T)$ onto $(\Bbb T, \mu_\lambda, R_\lambda)$.

It remains to  choose another $(C,F)$-sequence $(C_k',F_{k-1}')_{k\ge 1}$ with $\# C_k'=2$ for each $k\in\Bbb N$ in such a way that the associated $(C,F)$-system is isomorphic to $(X,\mu, T)$.
To this end, for each  $n\in\Bbb N$ and $k\in\{0,\dots,n-1\}$, we let 
$$
C_{n,k}:=\{0,2^kq_nh_{n-1}\}.
$$
It is easy to see that for each $n\in\Bbb N$,
$$
C_n= C_{n,0}+\cdots +C_{n,n-1}.
$$
We  now define the sequences $(C_k')_{k=1}^\infty$ and $(h_{k}')_{k=0}^\infty$
by listing their terms as follows:
\begin{gather*}
C_1, C_{2,0}, C_{2,1}, C_{3,0}, C_{3,1}, C_{3,2}, C_{4,0},\dots\qquad\text{ and}\\
h_0, 2q_1h_0,2q_2h_1, 2^2q_2h_1, 2q_3h_2, 2^2q_3h_2,2^3q_3h_2,\dots
\end{gather*}
respectively.
We also let $F_k':=\{0,1,\dots, h_k'-1\}$.
It is routine to check that~(\ref{2.1})--(\ref{2.3}) are satisfied for  $(C_k',F_{k-1}')_{k\ge 1}$.
Of course, the associated $(C,F)$-system $(X',\mu',T')$ is canonically isomorphic to 
$(X,\mu, T)$.\footnote{This follows from the fact that $(C_k,F_{k-1})_{k\ge 1}$ is a {\it telescoping} of $(C_k',F_{k-1}')_{k\ge 1}$. See \cite[\S1.4]{Da5} for details.}
Thus, Theorem~B is proved completely.
\end{proof}

\begin{proof}[Proof of Theorem C]
Fix a partition of $\Bbb N$ into infinitely many infinite subsets $\mathcal N_p$, $p\ge 2$.
Utilizing an inductive argument, we can construct a sequence $(C_n,F_{n-1})_{n=1}^\infty$ such that
 for each $n>0$,
\begin{enumerate}[label=\upshape(\roman*), leftmargin=*, widest=iii]
\item
$F_n=\{0,\dots,h_n-1\}$ for some $h_n>0$ and $\# C_n>1$,
\item
$F_n+F_n+C_{n+1}\subset F_{n+1}$,
\item
the sets $F_n-F_n+c-c'$, $c\ne c'\in C_{n+1}$, and $F_n-F_n$ are all mutually disjoint,
\item
$\# C_n\to\infty$ as $n\to\infty$ and
\item
$
\min_{1\le k\le n}\min_{f\ne f'\in F_k}|\lambda^{f-f'} - 1|> 2^{n+1}\max_{c\in C_n}
|1-\lambda^{c}|.
$
\item If $n\in\mathcal N_p$ for some $p\ge 2$ then there are two subsets $C_n^{(1)}$ and
$C_n^{(2)}$ of $C_n$ such that $\# C_n^{(1)}>0.3\# C_n$, $\# C_n^{(2)}>0.3\# C_n$,
every element of  $C_n^{(1)}$ is divisible by $p$, and  $c\equiv 1\pmod p$ for each $c\in C_n^{(2)}$.
\end{enumerate}
This sequence is constructed in an inductive way.
Suppose that we have already constructed a finite sequence $(C_k,F_k)_{k=1}^{n}$ satisfying (i)--(vi)  for some $n>1$. 
Our purpose is to define $C_{n+1}$ and $h_{n+1}$.
Fix $p>2$ such that $n+1\in\mathcal N_p$.
We now define inductively an auxiliary increasing sequence $(a(j))_{j=0}^\infty$ of non-negative integers.
Let $a(0):=0$.
Suppose that there is $l>0$ such that the elements $(a(j))_{j=1}^l$ have been defined and
\begin{enumerate}[label=\alph*)]
\item
the sets $F_n-F_n+a(i)-a(j)$, $i\ne j\in\{0,\dots,l\}$ and $F_n-F_n$ are mutually disjoint  and
\item
$
\min_{1\le k\le n}\min_{f\ne f'\in F_k}|\lambda^{f-f'} - 1|> 2^{n+1}\max_{1\le j\le l}
|1-\lambda^{a(j)}|.
$
\end{enumerate}

Consider separately 2 cases. 
Suppose first that $l+1$ is even.
Since the set $\{\lambda^{pq}\mid q\in\Bbb N\}$ is dense in $\Bbb T$, we can find $Q\in\Bbb N$ such that
\begin{itemize}
\item
$pQ>3(h_n+a(l))$ and 
\item
$\min_{1\le k\le n}\min_{f\ne f'\in F_k}|\lambda^{f-f'} - 1|> 2^{n+1}
|1-\lambda^{pQ}|.$
\end{itemize}
We then put $a(l+1):=pQ$.

Suppose now that $l$ is even. 
Since the set $\{\lambda^{pq+1}\mid q\in\Bbb N\}$ is dense in $\Bbb T$, we can find $Q\in\Bbb N$ such that
\begin{itemize}
\item
$pQ+1>3(h_n+a(l))$ and
\item
$\min_{1\le k\le n}\min_{f\ne f'\in F_k}|\lambda^{f-f'} - 1|> 2^{n+1}
|1-\lambda^{pQ+1}|.$
\end{itemize}
In this case put $a(l+1):=pQ+1$.

We note that in the two cases 
a) and b) hold with $l+1$ in place   of $l$.

Continuing these construction steps infinitely many times, we define an infinite sequence
$(a(j)_{j=0}^\infty$.
We now set $C_{n+1}:=\{a(0),a(1),\dots, a(n+1)\}$.
Choose $h_{n+1}$ large so that  (ii) is satisfied.
Thus, we defined $C_{n+1}$ and $F_{n+1}$ in a such a way that (i)--(v) are satisfied.
It remains to check that (vi) holds.
Let $C_{n+1}^{(1)}=\{a(j)\in C_{n+1}\mid \text{$j$ is even}\}$ and
$C_{n+1}^{(2)}=\{a(j)\in C_{n+1}\mid \text{$j$ is odd}\}$.
Then of course, $\# C_{n+1}^{(1)}>0.3\# C_{n+1}$ and 
$\# C_{n+1}^{(2)}>0.3\# C_{n+1}$, every element of  $C_{n+1}^{(1)}$ is divisible by $p$, and  $c\equiv 1\pmod p$ for each $c\in C_{n+1}^{(2)}$,
i.e. (vi) holds.
Thus, we defined the entire sequence $(C_n,F_{n-1})_{n=1}^\infty$ and (i)--(vi) hold for each $n\in\Bbb N$.

It follows from (i)--(iii)  that~(\ref{2.1})--(\ref{2.3}) are satisfied.
Therefore the associated $(C,F)$-system $(X,\mu, T)$ is well defined.
It follows from (ii) that $h_{n+1}>2h_{n}\# C_{n+1}$ for each $n>0$.
Hence $\mu(X)=\infty$ by (\ref{eq2.4}).
We deduce from (ii)--(iv) 
that $T$ is of zero type in view of Fact~\ref{fact2.2}.
Finally, (v) implies that the condition of Proposition~\ref{pr3} is satisfied.
Hence Proposition~\ref{pr3}
yields that $\lambda\in e(T)$ and every $\lambda$-eigenfunction $f:X\to\Bbb T$ of $T$ is one-to-one.
Let $\mu_\lambda':=\mu\circ f^{-1}$.
Then the dynamical system $(\Bbb T,\mu_\lambda', R_\lambda)$ is of rank one (with explicitly determined cutting-and-stacking parameters) and of zero type.
Hence, by the main result of \cite{RyTh},
$C(R_\lambda)=\{R_\lambda^n\mid n\in\Bbb Z\}$.
On the other hand, it is well known (and easy to verify) that
$$
C(R_\lambda)=\{R_\xi,\  \text{for }\xi\in\Bbb T\mid \mu_\lambda'\circ R_\xi=\mu_\lambda'\}.
$$
Moreover, since $R_\lambda$ is ergodic, it follows that for each $\xi\in\Bbb T$, we have that either
$\mu_\lambda'\circ R_\xi\sim\mu_\lambda'$ or 
$\mu_\lambda'\circ R_\xi \perp \mu_\lambda'$ \cite[9.22]{Na}.
We thus obtain that
$$
\Bbb T\setminus\{\lambda^n\mid n\in\Bbb Z\}
=\{\xi\in\Bbb T\mid \mu_\lambda'\circ R_\xi\perp \mu_\lambda'\}.
$$

Suppose that $e(T)$ contains a torsion  $\lambda$ of order $p\ge 2$.
Then by Corollary~\ref{finiteorder}, there exist $n>0$ and  subset $C_m^0\subset C_m$ for each $m\ge n$
such that $\# C_m^0>0.9\# C_m$ and $c\equiv c'\pmod p$ for all $c,c'\in C_m^0$.
This contradicts to~(vi).
Hence  $e(T)$ is torsion free.
Therefore
$T$ is totally ergodic.

It remains to prove the final claim of Theorem~C.
Let $\omega\in\Bbb T\setminus\{\lambda,\lambda^{-1}\}$ be of infinite order.
Suppose first that there are $n,m\in\Bbb Z$ such that $\lambda^n=\omega^m$.
Then $\mu'_\omega$ and $\mu'_\lambda$ are two measures on $\Bbb T$ that are invariant under
the same transformation $R_\lambda^n=R_\omega^m$.
Moreover, this transformation is ergodic with respect to each of these two measures because the systems 
 $(\Bbb T,\mu'_\lambda, R_\lambda)$ and $(\Bbb T,\mu'_\omega,R_\omega)$ are totally ergodic.
Hence either $\mu'_\omega\perp\mu'_\lambda$ or $\mu'_\omega\sim\mu'_\lambda$.
If the latter holds then  
$\{z\in\Bbb T\mid \mu'_\lambda\circ R_z\sim \mu'_\lambda\}=\{z\in\Bbb T\mid \mu'_\omega\circ R_z\sim\mu'_\omega\}$.
Hence $\{\lambda^l\mid l\in\Bbb Z\}=\{\omega^l\mid l\in\Bbb Z\}$.
This is only possible if either $\omega=\lambda$ or $\omega=\lambda^{-1}$.
Thus, we obtain a contradiction.
Hence $\mu'_\omega\perp\mu'_\lambda$, as claimed.

Suppose now that $\omega$ and $\lambda$ are independent and $\mu'_\omega\not\perp\mu'_\lambda$.
Then there is a Borel subset  $A\subset\Bbb T$ such that
$
\mu_\lambda'(A)>0 \text{ and }  (\mu_\lambda'\restriction A)\prec\mu'_\omega.
$
Then for each $n\in\Bbb N$,
$$
(\mu_\lambda'\restriction A)\circ R_\lambda^n\prec\mu'_\omega\circ R_\lambda^n\text{ and }
\mu'_\omega\circ R_\lambda^n\perp \mu'_\omega.
$$
Hence $
(\mu_\lambda'\restriction A)\circ R_\lambda^n\perp \mu'_\omega$, i.e. 
$(\mu_\lambda'\restriction R_\lambda^{-n}A)\perp \mu'_\omega$.
It follows that 
$$
\bigg(\mu_\lambda'\restriction \bigcup_{n>0}R_\lambda^{-n}A\bigg)\perp \mu'_\omega
$$
Since $\mu_\lambda'$ is nonatomic and ergodic with respect to $R_\lambda$, we obtain that
$ \bigcup_{n>0}R_\lambda^{-n}A=\Bbb T\,\pmod {\mu_\lambda'}$. 
Hence $\mu'_\omega\perp\mu'_\lambda$.
\end{proof}

A subset $\Lambda\subset\Bbb T$ is called {\it independent} 
if whenever $\lambda_{1}^{k_1}\lambda_{2}^{k_2}\ldots\lambda_{n}^{k_n}=1$
for some  $k_1,\dots, k_n\in\Bbb Z$, $\lambda_1,\dots,\lambda_n\in\Lambda$ and
 $n\in\Bbb N$ 
then $k_1=\cdots=k_n=0$.

\begin{remark}
Theorems B and  C can be strengthened in the following way.
Let $\Lambda$ be a countable  independent subset of $\Bbb T$.
Then there exists
\begin{itemize}
\item
 a rigid $(C,F)$-dynamical system  $(X,\mu, T)$ with explicitly defined sequence $(C_n,F_{n-1})_{n=1}^\infty$ such that $\mu(X)=\infty$, $\# C_n=2$ for each $n\in\Bbb N$, $\Lambda\subset e(T)$ end every $\lambda$-eigenfunction is one-to-one for each $\lambda\in\Lambda$;
 \item
  a zero type  $(C,F)$-dynamical system  $(X,\mu, T)$ with explicitly defined sequence $(C_n,F_{n-1})_{n=1}^\infty$ such that $\mu(X)=\infty$,
  $\Lambda\subset e(T)$ end every $\lambda$-eigenfunction is one-to-one for each $\lambda\in\Lambda$.
 \end{itemize}
The proof of these statements  is only a slight modification of the proof of~Theorems~B and C respectively. 
The proof is based on the well known fact that for each $\epsilon>0$, there is $n\in\Bbb N$ such that
$\sup_{\lambda\in\Lambda}|1-\lambda^n|<\epsilon$.
We leave details to the reader.
\end{remark}

\end{document}